\documentclass[12pt,reqno]{amsart}

\numberwithin{equation}{section} 

\usepackage{enumerate}
\usepackage{amssymb}
\usepackage{amsmath}
\usepackage{amscd}
\usepackage{amsthm}
\usepackage{amsfonts}
\usepackage{graphicx}
\usepackage[all]{xy}
\usepackage{verbatim}
\usepackage{hyperref}

\newtheorem{theorem}{Theorem}[section]

\newtheorem{lemma}[theorem]{Lemma}
\newtheorem{corollary}[theorem]{Corollary}
\theoremstyle{definition}\newtheorem{definition}[theorem]{Definition}

\theoremstyle{definition}\newtheorem{example}[theorem]{Example}
\newtheorem{conjecture}[theorem]{Conjecture}

\newtheorem{proposition}[theorem]{Proposition}

\newtheorem{problem}[theorem]{Problem}
\theoremstyle{definition}
\theoremstyle{definition}\newtheorem{remarks}[theorem]{Remarks}
\theoremstyle{definition}
\theoremstyle{definition}\newtheorem*{acknowledgments}{Acknowledgments}
\theoremstyle{definition}\newtheorem*{question}{Question}
\newcommand{\al}{\alpha}
\newcommand{\be}{\beta}
\newcommand{\ga}{\gamma}
\newcommand{\Ga}{\Gamma}
\newcommand{\del}{\delta}

\newcommand{\lam}{\lambda}

\newcommand{\eps}{\epsilon}

\newcommand{\Om}{\Omega}

\newcommand{\cP}{\mathcal{P}}

\newcommand{\bC}{\mathbb{C}}

\newcommand{\bR}{\mathbb{R}}
\newcommand{\bZ}{\mathbb{Z}}
\newcommand{\bQ}{\mathbb{Q}}

\newcommand{\bN}{\mathbb{N}}

\newcommand{\bT}{\mathbb{T}}

\newcommand{\bk}{\textbf{k}}

\newcommand{\bm}{\textbf{m}}

\newcommand{\bp}{\textbf{p}}
\newcommand{\bq}{\textbf{q}}

\newcommand{\bx}{\textbf{x}}
\newcommand{\by}{\textbf{y}}

\newcommand{\SL}{\operatorname{SL}}

\newcommand\norm[1]{\left\|#1\right\|}
\newcommand\set[1]{\left\{#1\right\}}
\newcommand\pa[1]{\left(#1\right)}

\newcommand\idist[1]{\langle#1\rangle}
\newcommand\av[1]{\left|#1\right|}
\newcommand\on[1]{\operatorname{#1}}
\newcommand\diag[1]{\operatorname{diag}\left(#1\right)}

\newcommand{\onto}{\xymatrix{\ar@{>>}[r]&}}
\newcommand{\da}[4]{\xymatrix{#1 \ar@<.5ex>[r]^{#2} \ar@<-.5ex>[r]_{#3} & #4}}

\begin{document}
\title{Grids with dense values}
\author{Uri Shapira}
\thanks{* The author acknowledges the partial support of ISF grant number 1157/08}
\begin{abstract}
Given a continuous function from Euclidean space to the real line, we analyze (under some natural assumption on the function), the set of values
it takes on translates of lattices. Our results are of the flavor: For almost any translate, the set of values is dense in the set of possible values. The results are then applied
to a variety of concrete examples, obtaining new information in classical discussions in different areas in mathematics; in particular, Minkowski's conjecture regarding products of inhomogeneous forms and inhomogeneous Diophantine approximations.
\end{abstract}
\maketitle
\section{Introduction}\label{introduction}
Given a continuous function $F:\bR^d\to \bR$, it is a natural question in number theory, to try and analyze the set of values $F$ takes on points 
of a lattice in $\bR^d$. Up to a linear change of variable, this question is of course equivalent to analyzing the values $F$ takes at integer points.
We shall be interested in an inhomogeneous variant of this discussion; we try to analyze the set of values $F$ takes on grids, that is on translated
lattices. We approach this discussion from a dynamical point of view which leads us to impose some natural assumptions on the function $F$ under 
consideration. We present a variety of concrete examples in \S~\ref{main results} and applications to classical discussions in Diophantine approximations and the geometry of numbers in \S~\ref{app 1}. The main tools we develop to derive our results, are the Mixing and the Coset lemmas, appearing in \S~\ref{tools}. The discussion in \S~\ref{tools} is concerned with closures of certain random sequences on the $d$-torus. It is independent of the rest of 
this paper and may be of independent interest on its own.  
\subsection{Basic notions}
The basic objects we shall work with are lattices and grids. As we wish to exploit dynamical methods, it is most convenient to present spaces, the points of which are lattices and grids.  	
We set $G=\SL_d(\bR), \Ga=\SL_d(\bZ)$, and let $X_d=G/\Ga$. A point $g\Ga\in X_d$ will be denoted by $\bar{g}$. The space $X_d$ can be identified with the space of unimodular (i.e.\ of covolume 1) lattices in $\bR^d$. The identification is defined by associating to $\bar{g}\in X_d$ the lattice spanned by the columns of $g$. We let $Y_d=\pa{G\ltimes \bR^d}/\pa{\Ga\ltimes \bZ^d}$. The space $Y_d$ is identified with the space of unimodular \textit{grids} in $\bR^d$, where by a grid we mean a set of 
the form $\bar{g}+v=\set{w+v:w\in\bar{g}}.$ Points of $Y_d$ are denoted by $\bar{g}+v$ where $g\in G$ and $v\in\bR^d$. We endow $X_d,Y_d$ with the quotient topology, induced from the usual topology on $G$ and $G\ltimes \bR^d$ respectively. In fact, $X_d,Y_d$ are smooth manifolds and inherit a complete Riemannian metric from the corresponding group. It will be convenient to interchangeably think of the points of $X_d,Y_d$, as subsets of $\bR^d$, or as points in the corresponding manifold. We shall denote points of $X_d,Y_d$ by lower case letters, $x, x_1,y,y'$ etc. It is a good exercise
for the reader who is not familiar with the topologies introduced above, to work out the meaning of two lattices (resp.\ grids) being close to each other, when thinking of them as subsets in $\bR^d$. This boils down to saying that in a very large box, centered at the origin, the corresponding two sets are close in the usual sense. 

There is an obvious projection $\pi:Y_d\to X_d$ given by $\pi(\bar{g}+v)=\bar{g}$. The fiber $\pi^{-1}(x)$, for $x\in X_d$, is naturally identified with the torus $\bR^d/x$. Note also that $X_d$ naturally embeds in $Y_d$. 
For a continuous function $F:\bR^d\to \bR$, we define for each $y\in Y_d$ the \textit{value set}
\begin{equation}\label{value set}
V_F(y)=\set{F(v):v\in y}.
\end{equation}
In general, it is of interest to analyze for a fixed grid $y\in Y_d$, the value set $V_F(y)$. In particular, one would like to answer questions such as: Is
the value set dense, discrete, or does it contain zero in its closure. We make the following definition which will be of most interest to us in this paper.
\begin{definition}\label{DV} 
Consider a fixed continuous function $F:\bR^d\to\bR$.
\begin{enumerate}
	\item A grid $y\in Y_d$, is DV (dense values), if the value set, $V_F(y)$, is dense in the image $F(\bR^d)$.
	\item A lattice $x\in X_d$ is grid-DV, if all its grids are DV.
	\item Given a lattice $x\in X_d$, and a probability measure $\mu$, on the torus $\pi^{-1}(x)$, we say that $x$ is $\mu$-almost surely grid-DV, if 
	$\mu$-almost any grid $y\in\pi^{-1}(x)$, is DV (we sometime express this by saying that $x$ is almost surely grid-DV with respect to $\mu$) .
\end{enumerate}
\end{definition} 
In general, it is a hard problem to decide for a specific function $F$, whether or not there exist grid-DV lattices, although in some concrete examples the answer is known (see \S~\ref{main results}). On the other hand, the existence of almost surely grid-DV lattices (with respect to some natural family of measures) is guaranteed, once some reasonable assumptions on $F$ are imposed.  
In this paper we find sufficient conditions for a lattice $x\in X_d$ to be almost surely grid-DV with respect to some measures which we now turn to describe. 
\subsection{Haar measures of subtori}\label{H-S-T}
Let $x\in X_d$ be given, we say that a subspace $U<\bR^d$, is rational with respect to $x$, if $x\cap U$ is a lattice in $U$. The rational subspaces are in one to one correspondence with the closed connected subgroups of the torus $\pi^{-1}(x)$. We refer to these as \textit{subtori} and denote the subtorus corresponding to a rational subspace, $U$, by $U+x$. 
The Haar probability measure on such a subtorus is denoted by $\lam_U$. For $v\in\bR^d$ and a subtorus $U+x$, we denote by $v+(U+x)$ the coset of 
$v$ with respect to the subtorus. The translation of the Haar measure $\lam_U$ by $v$, supported on the coset, is denoted by $\lam_{v+U}$. When $U=\bR^d$, we sometime refer to $\lam_U$ as Lebesgue measure. 
\subsection{The invariance group}
Let  
\begin{align*}
\textbf{H}_F&=\set{g\in G: F\circ g=F},\\
H_F&=\textbf{H}^\circ_F.
\end{align*}
Here $\textbf{H}^\circ_F$ denotes the connected component of the identity in $\textbf{H}_F$. The group $H_F$ will be referred to as the \textit{invariance group} of $F$. It is a closed subgroup of $G$ and in most interesting cases it is (the connected component of the identity of ) an algebraic group. This happens when $F$ is a polynomial, but also in some other cases too.
We say that $F$ is \textit{noncompact} if $H_F$ is noncompact . As our approach is dynamical, we shall only discuss the noncompact case. 

There are natural actions (by left translation) of $G$ on $X_d,Y_d$. These actions commute with $\pi$. When thinking of the 
points of $X_d,Y_d$ as subsets of $\bR^d$, these actions are induced by the linear action of $G$ on $\bR^d$. 
The importance of the invariance group to our discussion is that it leaves the value sets invariant; for any $y\in Y_d$ and any
$h\in H_F$,
\begin{align}\label{inv 1}
V_F(y)=V_F(hy).
\end{align}
The following lemma illustrates how dynamics comes into the game. The reader should deduce it from equation ~\eqref{inv 1}, the continuity of $F$, 
and the topology of $Y_d$:
\begin{lemma}[Inheritance]\label{inheritance}
Let $y,y_0\in Y_d$ be such that $y_0\in\overline{H_Fy}$. Then $\overline{V_F(y_0)}\subset \overline{V_F(y)}.$
In particular, if $y_0$ is DV, then so is $y$. If $y$ has a discrete value set, then so does
$y_0$.
\end{lemma}
As $\SL_d(\bZ)$ is a lattice in $\SL_d(\bR)$, the space $X_d$ carries a $G$-invariant probability measure. Similarly, $Y_d$ carries a $G\ltimes\bR^d$-invariant probability measure. We shall refer to both these measures as the Haar measures. It is not hard to show that the 
Haar measure on $Y_d$ disintegrates to the Lebesgue measures
on the fibers $\set{\pi^{-1}(x):x\in X_d}$, with respect to the Haar measure on $X_d$. The Howe-Moore theorem asserts in this case, that the $G$-actions on $X_d,Y_d$ are mixing with respect to these measures. As we assume that $H_F$ is noncompact, it acts mixingly and in particular, ergodically on both spaces. It follows in particular, that for almost any grid $y\in Y_d$, $H_Fy$ is dense in $Y_d$. This gives us the following immediate corollary
of the inheritance lemma
\begin{corollary}\label{ergo}
Almost any $y\in Y_d$ is DV. Moreover, almost any lattice, $x\in X_d$, is almost surely grid-DV with respect to the Lebesgue measure of the full
torus $\pi^{-1}(x)$.
\end{corollary}
In this paper we wish to point out a fairly general connection between the dynamical behavior of a lattice $x\in X_d$ under the action of the 
invariance group, and the value set, $V_F(y)$, along a fiber $\pi^{-1}(x)$. Our aim is to sharpen Corollary ~\ref{ergo} and to develop better understanding of the set of almost surely
grid-DV lattices. 

Before ending this introduction, let us make the following definitions which we need in order to state our results. 
We say that a sequence in a topological space is divergent, if it has no converging subsequences. An orbit $H_Fx$, is said to be divergent if for any divergent sequence $h_n\in H_F$, the sequence $h_nx$ is divergent in $X_d$ (i.e.\ if the orbit map is proper).  
\begin{definition}\label{nondeg}
A continuous function $F:\bR^d\to\bR$ is \textit{nondegenerate} if for any nontrivial subspace $U<\bR^d$ and any grid $y\in Y_d$, one has
\begin{align}\label{should be}
\set{F(u+v):u\in U,v\in y}=F(\bR^d).
\end{align}
\end{definition}
A typical example of a degenerate function, which we wish to avoid, is a polynomial with coefficients in $\bZ$, which does not depend on one of the variables (the simplest example is $F:\bR^2\to\bR$, given by $F(v_1,v_2)=v_1$). In this case, choosing the grid to be simply $\bZ^d$ and $U$ to be the 
line corresponding to the variable that does not appear in $F$, we see that the set in ~\eqref{should be}, is discrete.
\begin{definition}\label{a.f}
A sequence $h_n\in G$ is said to be \textit{almost finite} with respect to a subspace $U<\bR^d$, if there exist a  a sequence 
$\eps_n\in G$ with $\eps_n\to e$, such that the set of restrictions $\set{\eps_nh_n|_U}$ is finite.
\end{definition}
Note that any diverging sequence, $h_n$, is not almost finite with respect to $\bR^d$. Also, if $h_n$ are diagonalizable, and the eigenvalues approach
either $0$ or $\infty$, then $h_n$ is not almost finite with respect to any nontrivial subspace.
\section{Examples and results}\label{main results}
\subsection{Main theorem} The following theorem is a simplified version of our main result. This version is sufficiently strong for most of our applications.
\begin{theorem}\label{main cor 1}
Let $F:\bR^d\to \bR$ be nondegenerate and noncompact. Let $x\in X_d$ be a lattice with a nondivergent $H_F$-orbit and let $\lam$ be the Lebesgue measure
of the full torus $\pi^{-1}(x)$. Then $x$ is $\lam$-almost surely grid-DV.
\end{theorem}
Theorem ~\ref{main cor 1} follows from the following theorem which is the main result in this paper. It is  proved in \S~\ref{pfs}. 
\begin{theorem}\label{main theorem}
Let $F:\bR^d\to \bR$ be nondegenerate and noncompact. Let $x\in X_d$ be a lattice with a nondivergent $H_F$-orbit and let $\lam_{w+U}$ be the translation of the Haar measure supported on the subtorus $U+x$ by a vector $w\in\bR^d$. Then, if there exists a divergent sequence $h_n\in H_F$, such that $h_nx$ converges and $h_n$ is not almost finite with respect to $U$, then $x$ is $\lam_{w+U}$-almost surely grid-DV.  
\end{theorem} 
In some examples (as will be seen below), the fact that the sequence $h_n$ is not almost finite follows automatically from properties of the invariance
group. 
%
%
%
%
We now turn to apply these results to a variety of concrete examples.
\subsection{Rank one}\label{rank one}
The following family of examples is particularly relevant to Diophantine approximations (see \S\ref{app 2}). Let $m,n$ be positive integers such that $m+n=d$. Let us denote vectors in $\bR^d$ as column vectors $(\bx,\by)^t$ (here $t$ stands for transpose), where $\bx\in\bR^n,\by\in\bR^m$. Let $P_{n,m}:\bR^d\to\bR$ be defined by $P_{n,m}(\bx,\by)=\norm{\bx}_2^n\norm{\by}_2^m$ (where we denote here by $\norm{\cdot}_2$, the Euclidean norm on the corresponding space). Denoting the invariance group by $H_{n,m}$, it is not hard to see that
$H_{n,m}=\on{SO}(n)\times\on{SO}(m)\times\set{a_{n,m}(t)}_{t\in\bR}$, where 
\begin{equation}\label{anm}
a_{n,m}(t)=\diag{\underbrace{e^{mt},\dots,e^{mt}}_n,\underbrace{e^{-nt},\dots,e^{-nt}}_m}.
\end{equation}
It is not hard to see that $P_{n,m}$ is nondegenerate and noncompact, hence Theorem ~\ref{main theorem} applies. In fact, we have the
following theorem.
\begin{theorem}\label{rank 1}
Let $x\in X_d$ be a lattice with a nondivergent $H_{n,m}$-orbit. Then for any subspace $\set{0}\ne U<\bR^d$, rational with respect to $x$, and any $w\in \bR^d$, $x$ is $\lam_{w+U}$-almost surely grid-DV.
\end{theorem}
\begin{proof}
Let $\lam_{w+U}$ be a measure as in the statement. As the orbit $H_{n,m}x$ is nondivergent, there exists a diverging sequence $h_n\in H_F$, such that
$h_nx$ converges. In fact, as $H_{n,m}$ is the product of a compact group with the one parameter group $a_{n,m}(t)$, we may assume that $h_n=a_{n,m}(t_n)$ for some 
$t_n\to \pm\infty$. As the eigenvalues of $a_{n,m}(t_n)$ approach $0$ or $\infty$, it follows that $h_n$ is not almost finite with respect to $U$. Theorem 
~\ref{main theorem} applies and we conclude that $x$ is $\lam_{w+U}$-almost surely grid-DV.
\end{proof}
In  \S\ref{app 2} we apply Theorem ~\ref{rank 1} to Diophantine approximations and in particular to derive new results on nonsingular forms.
We also remark here that because $H_{n,m}$ is of real rank 1, for any lattice $x\in X_d$, there always exist grids which are not DV; that is, there are no grid-DV lattices in this case. This result was proved first by Davenport ~\cite{D} for the case $d=2, n,m=1$, while recently, Einsiedler and Tseng proved (for general $d$, $n=d-1, m=1$) ~\cite{ET} that the set of such grids is of full Hausdorff dimension, and in fact, a winning set for Schmidt's game (see also ~\cite{BFK} for generalizations).  
\subsection{The product of linear forms}\label{product of linear forms}
Let $N:\bR^d\to \bR$ be the function $N(x)=\prod_1^dx_i$. We denote the invariance group $H_N$, by $A$. It is the group of diagonal matrices with positive diagonal entries in $G$. It is not hard to see that $N$ is nondegenerate and noncompact. Hence, Theorem ~\ref{main theorem} and Theorem ~\ref{main cor 1} apply and we know that any lattice with a nondivergent $A$-orbit is almost surely grid-DV with respect to the Haar measure of the full torus. We shall see in the next section that this result
has significant implications towards Minkowki's conjecture. It is worth noting here 
that a classification of the divergent $A$-orbits, due to Margulis, is given in \cite{TW}. It is proved there that $Ax$ is divergent, if and only if, there exists 
$a\in A$, such that the lattice $ax$ is contained in $\bQ^d$.
There are many extra assumptions one can impose on a lattice $x$ to ensure that  $x$ is 
almost surely grid-DV with respect to any nontrivial measure of the form $\lam_{w+U}$. The following two corollaries
are examples of such. The first is proved in \S~\ref{pfs}.
\begin{corollary}\label{product 1}
Let $x\in X_d$ be a lattice which does not contain any vector lying on the hyperplanes of the axes (that is the hyperplanes orthogonal to the standard
basis vectors). Then for any subspace $\set{0}\ne U<\bR^d$, rational with respect to $x$, and for any $w\in\bR^d$, $x$ is $\lam_{w+U}$-almost surely grid-DV.
\end{corollary} 
\begin{corollary}\label{product 2}
Let $x\in X_d$ be a lattice such that there exists a sequence $$a_n=\diag{e^{t_1^{(n)}},\dots,e^{t_d^{(n)}}}\in A,$$
such that $a_nx$ converges,
and for any $1\le i\le d$, $t_i^{(n)}$ diverges. Then $x$ is almost surely grid-DV with respect to any nontrivial measure $\lam_{w+U}$
\end{corollary}
\begin{proof}
This follows immediately from Theorem ~\ref{main theorem}, as under our assumptions, the sequence $a_n$ is not almost finite with respect to any
nontrivial subspace.
\end{proof}
Thanks to the classification of divergent $A$-orbits, we have the following theorem which shows what happens for lattices with divergent orbits.
\begin{theorem}\label{disc products}
Let $x\in X_d$ be a lattice with a divergent $A$-orbit. Then $V_N(y)$ is discrete (and moreover $\overline{V_N(y)}=V_N(y)$), for any grid $y\in\pi^{-1}(x)$.
\end{theorem}
\begin{proof}
As the value set does not change along the orbit, we conclude from the classification of divergent $A$-orbits, mentioned above, that we may assume 
that $x\subset \bQ^d$. In fact, it is not hard to see that the validity of the theorem is stable under commensurability; that is, if $x,x'$ are commensurable lattices (i.e.\ $x\cap x'$ is of finite index in both), then the sets $V_N(y)$ are discrete, for any $y\in \pi^{-1}(x)$, if and only if
the same is true for any $y\in\pi^{-1}(x')$. This enables us to assume that $x=\bZ^d$. So, we are left to verify that for any vector $v\in\bR^d$, the set $$V_N(v+\bZ^d)=\set{\prod_1^d(n_i+v_i):n_i\in\bZ},$$
is discrete. To prove this, let $\prod_1^d(n_i^{(k)}+v_i)$ be a converging sequence of elements of $V_N(v+\bZ^d),$ and assume by way or contradiction 
that its elements are distinct (this rules out the possibility of having $n_i^{(k)}+v_i=0$ for more than one $k$). For each $i$, the sequence $(n_i^{(k)}+v_i)$, is discrete. Then we are able to take a subsequence so that for each $i$, along the 
subsequence, either $(n_i^{(k)}+v_i)$ is constant and nonzero, or it diverges. As we assume the original sequence is converging, we must have that the subsequence
is constant from some point which contradicts our assumption.
\end{proof}
In contrast to the the situation presented in the previous subsection (Theorem ~\ref{rank 1}), where nondivergence of the orbit of a lattice under the invariance group was sufficient to ensure that the lattice is almost surely grid-DV with respect to any measure of the form $\lam_{w+U}$, we work out the following example which shows the existence of a lattice $x\in X_3$, with a nondivergent $A$-orbit, and a nontrivial subspace $U<\bR^3$, rational with respect to $x$, such that $x$ is not $\lam_U$-almost surely grid-DV. In fact, we shall see that the value set $V_N(x+u)$ is discrete for any $u\in U$ 
\begin{example}\label{example}
In this example, as we will mix the dimensions, we denote our function $N$ and its invariance group $A$ in dimension $d$ by $N_d$ and $A_d$ respectively.
Let $g_0\in\SL_2(\bR)$ be such that the lattice $\bar{g}_0$ has a compact orbit in $X_2$, under the action of the group $A_2=\set{\diag{e^t,e^{-t}}}$.
It is well known that this is equivalent to saying that the value set $V_{N_2}(\bar{g}_0)$ is discrete and does not contain zero. Let $g=\pa{\begin{array}{ll}
1&0\\
0&g_0
\end{array}}\in \SL_3(\bR)$ (here the zeros stand for the corresponding row and column zero vector in $\bR^2$), and let $U=\set{(*,0,0)^t\in\bR^3}$. Then it is clear
that for any
grid $y=\bar{g}+u$, $u\in U$, the value set $V_{N_3}(y)$, is discrete. Indeed, if $u=(\al,0,0)^t$, then 
$$V_{N_3}(\bar{g}+u)=\set{(n+\al)N_2(w):n\in\bZ,w\in\bar{g}_0},$$
which is clearly discrete as $V_{N_2}(\bar{g}_0)$ is.

To see how this fits with Theorem ~\ref{main theorem}, note that if 
$$a_n=\diag{e^{t_1^{(n)}},e^{t_2^{(n)}},e^{-(t_1^{(n)}+t_2^{(n)})}}$$
is a sequence in $A_3$, satisfying that $a_n\bar{g}$ converges in $X_3$, then it is easy to see that $t_1^{(n)}$ must converge, and so the sequence $a_n$ is almost finite with respect to $U$.
\end{example}
We remark here that for $d\ge 3$, the fact that the invariance group is of higher rank, enabled the author of the present paper to prove ~\cite{Shap1} the existence
of grid-DV lattices. In this case such lattices are also known as GDP lattices (grid-dense-products).
\subsection{Indefinite quadratic forms}
Let $p,q>0$ be integers such that $d=p+q\ge 3$. Let $F$ be the indefinite quadratic form given by $F(x)=\sum_1^px_i^2-\sum_{p+1}^{p+q}x_i^2$.
In this case 
$H_F=\on{SO}(p,q)$. It is not hard to see that $F$ is nondegenerate (in the sense of Definition ~\ref{nondeg}) and 
noncompact, hence Theorem ~\ref{main theorem} applies. In fact, as $H_F$ is generated by unipotents, it follows from \cite{Margulis-non-div} that $H_F$ has no divergent orbits and so the theorem applies for any lattice.

The fact that $H_F$ is generated by unipotents allows one to obtain much sharper results than Theorem ~\ref{main theorem}. In fact, one has the following 
theorem and corollary which follow from Ratner's orbit closure theorem and Lemma ~\ref{inheritance}. We do not elaborate on the arguments as this case
is well understood and sharper results than those stated here are available (see for example ~\cite{MM}). 
\begin{theorem}\label{quadratic 1}
Let $y\in Y_d$ be given. If the orbit $H_Fy$ is closed, then the value set $V_F(y)$ is discrete. Otherwise, the value set is dense. 
\end{theorem} 
It is not hard to see that in this case, a grid $y\in Y_d$ has a closed $H_F$-orbit, if and only if, $\pi(x)\in X_d$ has a closed orbit, and $y$ is \textit{rational}; that is to say $y=x+v$ for $v\in\on{Span}_\bQ{x}$. So we have a full understanding:
\begin{corollary}\label{quadratic 2}
Let $x\in X_d$ be given. If the orbit, $H_Fx$, is not closed, then $x$ is grid-DV. If the orbit, $H_Fx$, is closed, then rational grids have discrete 
value sets, while irrational grids are DV. In any case, $x$ is almost surely grid-DV with respect to any nontrivial measure $\lam_{w+U}$.
\end{corollary}
\section{Applications}\label{app 1}
We now apply some of the above theorems. In \S\ref{minko} we apply Theorem ~\ref{main cor 1} to derive 
new information towards Minkowski's conjecture. In \S\ref{app 2}, we apply Theorem ~\ref{rank 1}, to generalize some results on inhomogeneous approximations and provide a partial answer to a problem posed in ~\cite{Tseng-shrinking}. 
\subsection{Minkowski's conjecture}\label{minko}
We shall use the notation of \S\ref{product of linear forms}. The following conjecture is usually attributed to Minkowski and hence named after him. It has remained open for over a century.
\begin{conjecture}[Minkowski]
Let $d\ge 2$ and $N:\bR^d\to \bR$ be the function $N(x)=\prod_1^dx_i$. Then, for any $y\in Y_d$ one has $V_N(y)\cap [-2^{-d},2^{-d}]\ne\emptyset$. 
\end{conjecture}
To this date, Minkowski's conjecture is verified up to dimension 7. For more information about the interesting history and recent developments we 
refer the reader to the recent papers ~\cite{Mc}, ~\cite{HG}, where it is proved for dimensions 6 and 7. We say that \textit{a lattice, $x\in X_d$, 
satisfies Minkowski's conjecture}, if any one of its grids satisfies it. In ~\cite{Bom}, Bombieri proved that for $d\ge 5$, for any lattice $x\in   X_d$, the set $\set{y\in\pi^{-1}(x):y\textrm{ satisfies Minkowski's conjecture}}$, has Lebesgue measure $\ge 2^{-\frac{n+1}{2}}$. The results of Narzullaev ~\cite{Narzu}, allows one to strengthen Bombieri's result, still obtaining a poor lower bound. We strengthen this 
estimate to the best possible from the measure theoretic point of view and prove
\begin{theorem}\label{app Mink}
For any $d\ge 2$ and any lattice, $x\in X_d$, the set 
\begin{align}\label{set 1}
\set{y\in\pi^{-1}(x):y\textrm{ does not satisfy Minkowski's conjecture}},
\end{align}
is of Lebesgue measure zero. 
\end{theorem}
\begin{proof}
Let $x\in X_d$ be given. We divide the argument into cases. If the orbit $Ax$, of the invariance group, is nondivergent, Theorem ~\ref{main cor 1} implies that almost any
grid $y$ of $x$ is DV. In particular, $V_N(y)\cap[-2^{-d},2^{-d}]\ne\emptyset$ and the claim follows. If the orbit $Ax$ is divergent, then thanks 
to the characterization of divergent orbits given in ~\cite{TW}, we see that there exists $a\in A$, such that the lattice $ax$ has a basis consisting of vectors with rational coordinates. We refer to such lattices as \textit{rational}. It is well known, (see ~\cite{Mac}), that rational lattices satisfy
Minkowski's conjecture, hence the set in ~\eqref{set 1} is in fact empty. 
\end{proof}
We note here that if we impose on the lattice $x$ the further assumptions of Corollaries ~\ref{product 1},~\ref{product 2}, then the above theorem
can be strengthened to say that the set of possible counterexamples to Minkowski's conjecture in the torus $\pi^{-1}(x)$, has measure zero with respect
to any measure of the form $\lam_{w+U}$.
\subsection{Inhomogeneous approximations and nonsingular forms}\label{app 2}
In this section we apply Theorem ~\ref{rank 1} to deduce some new results about nonsingular forms. We fix some positive integers $n,m$ and consider a matrix $A\in\on{Mat}_{n\times m}(\bR)$. In Diophantine approximations, $A$ is usually regarded as defining a system of $n$ linear forms in $m$ variables, hence we shall refer to such a matrix as a \textit{form}. In first reading it might be useful to take $m=1$ and $n$ arbitrary, hence $A\in\bR^n$ stands for a column vector, or even $n=m=1$, and then $A$ stands for a real number.

For any dimension $k$, let $\idist{\cdot}$ denote the distance from $\bZ^k$ in $\bR^k$, i.e.\ for $v\in\bR^k$, $$\idist{v}=\min\set{\norm{v-\bq}_2:\bq\in\bZ^k}.$$
Note that for $v\in\bR^k$, $\idist{v}$ can be thought of as the distance in the torus $\bT^k=\bR^k/\bZ^k$, from $v$ to $0$. Here and in the sequel we abuse notation freely and denote points on the $k$-torus the same as their representatives in $\bR^k$.  
\begin{definition}
A form $A\in\on{Mat}_{n\times m}(\bR)$ is \textit{singular}, if for any $\del>0$, for any large enough $N\in\bN$, one can find $\bq\in\bZ^m$ such that
\begin{equation}\label{si-ngul}
0<\norm{\bq}_2\le N \textit{ and } \idist{A\bq}<\frac{\del}{N^{\frac{m}{n}}}.
\end{equation}
Following Baker ~\cite{Bakerone}, we say that the form $A$ is \textit{highly singular}, if there exists $\eps>0$ such that for all large enough $N\in\bN$, one can find $\bq\in\bZ^m$ such that 
\begin{equation}
\norm{\bq}_2\le N \textit{ and } \idist{A\bq}<N^{-\frac{m}{n}-\eps}.
\end{equation}
\end{definition} 
It is clear that a highly singular form is singular.
By applying Theorem ~\ref{rank 1}, we shall deduce some results about nonsingular forms. We shall prove Theorem ~\ref{app 4}, but for the meantime let us formulate a restricted version of it in the form of the following theorem:
\begin{theorem}\label{simple version}
Let $A\in\on{Mat}_{n\times m}(\bR)$ be a nonsingular form. Then, for Lebesgue almost any $\bx\in\bR^n$, the following set is dense in $[0,\infty)$,
\begin{equation}\label{dense set}
\set{\norm{\bq}_2^m\idist{A\bq+\bx}^n:\bq\in\bZ^m}.
\end{equation}
In particular, the following statement holds
\begin{equation}\label{innovation}
\textrm{For Lebesgue almost any }\bx\in\bR^n,\;\; \liminf_{\bq\in\bZ^m}\norm{\bq}_2^\frac{m}{n}\idist{A\bq+\bx}=0.
\end{equation}
\end{theorem}
\begin{remarks}
\begin{enumerate}
\item It follows from the main theorem in ~\cite{B-L}, that for any nonsingular form $A\in\on{Mat}_{n\times m}(\bR)$,  the following statement holds
\begin{equation}\label{bug lor1}
\textrm{For Lebesgue almost any }\bx\in\bR^n,\; \liminf_{\bq\in\bZ^m}\norm{\bq}_2^w\idist{A\bq+\bx}=\Bigg\{\begin{array}{ll}0&\textrm{for }w<\frac{m}{n}\\ \infty& \textrm{for }w>\frac{m}{n}.\end{array}
\end{equation}
The innovation in Theorem \ref{simple version}, is that the $\liminf$ in the critical exponent $\frac{m}{n}$, equals zero (and in fact much more, namely, the set in ~\eqref{dense set} is dense).
\item The above theorem generalizes a series of results. When one takes $n=m=1$, hence $A$ is a real number, Kurzweil proved ~\cite{Kurz} that for almost any $A$, ~\eqref{innovation} holds. Not long ago Kim proved ~\cite{Kim} that Kurzweil's result holds for any irrational $A\in\bR$ (which is the same as being nonsingular in these dimensions). Tseng reproved Kim's result ~\cite{Tseng-shrinking}, and raised the following question, aiming to generalize Kim's result to higher dimensions:
\begin{question}[Tseng]
For $m=1$, hence $A\in\bR^n$ is a column vector, is it true that if $A$ does not lie in a rational hyperplane, then ~\eqref{innovation} holds.
\end{question}
Galatolo and Peterlongo ~\cite{G-P} constructed a (singular) vector in $\bR^2$, giving a negative answer to Tseng's question. Nevertheless, Theorem ~\ref{simple version} tells us that the answer to Tseng's question is positive for nonsingular vectors (a negative answer to Tseng's question could also be presented by a vector $A\in\bR^n$ which does not lie in a rational hyperplane such that the transposed form $A^t$ is highly singular. To see this, note that from the main result in ~\cite{B-L}, it follows that in this case the value of the $\liminf$ in ~\eqref{innovation} equals $\infty$ for almost any $\bx$).
\end{enumerate}
\end{remarks}

We turn now to the statement and proof of Theorem ~\ref{app 4}. This theorem implies Theorem ~\ref{simple version} when applied with respect to  Lebesgue measure but is in fact considerably stronger. Nevertheless, it is a simple
application of Theorem ~\ref{rank 1}. In order to state this theorem, we need to link the above discussion with the discussion of \S\ref{rank one}, namely the discussion of values of forms on grids. Let $d=n+m$ and recall the notation of \S\ref{rank one}: We write vectors in $\bR^d$ as $(\bx,\by)^t$, where $\bx\in\bR^n$ and $\by\in\bR^m$. We let $P_{n,m}:\bR^{d}\to\bR$ be the function given by $P_{n,m}(\bx,\by)=\norm{\bx}_2^n\cdot\norm{\by}_2^m$, and denote the invariance group of it by $H_{n,m}$. Recall also that $H_{n,m}=\on{SO}_n(\bR)\times\on{SO}_m(\bR)\times \set{a_{n,m}(t)}_{t\in\bR}$, where $a_{n,m}(t)$ is given by ~\eqref{anm}.
For a form $A\in \on{Mat}_{n\times m}(\bR)$ we denote 
\begin{equation}\label{g_v}
g_A=\pa{\begin{array}{ll}
I_n& A\\
0&I_m
\end{array}}\in G,
\end{equation}
where $I_n,I_m$ denote the identity matrices of the corresponding dimensions and $0$, the zero matrix in $\on{Mat}_{m\times n}(\bR)$. We denote the lattice $\bar{g}_A\in X_d$, spanned by the columns of $g_A$, by $x_A$. The following result of Dani is well known (see Theorem 2.14 in ~\cite{Dani85}), 
and furnishes the link between the nondivergence condition in Theorem ~\ref{rank 1} and the nonsingularity property. We include the proof for the completeness of our presentation.
\begin{lemma}\label{n-d}
A form $A\in \on{Mat}_{n\times m}(\bR)$ is singular, if and only if, the orbit $H_{n,m}x_A$ is divergent in $X_d$.
\end{lemma}
\begin{proof}
Assume that the orbit $H_{n,m}x_A$ is divergent. As $a_{n,m}(t)$ is cocompact in $H_{n,m}$, this is equivalent to saying that the $a_{n,m}(t)$-orbit of $x_A$ is divergent.  Fix $\del>0$. We need to show that for all large enough $N\in\bN$, one can solve ~\eqref{si-ngul}. The divergence of the orbit implies, by Mahler's compactness criterion, that there exists $T_0\in\bR$, such that for any $t>T_0$, the lattice
$a_{n,m}(t)x_A$ contains a nonzero vector of length $<\del$. I.e.\ for any $t> T_0$, there exist $\bp\in\bZ^n,\bq\in\bZ^m$, not both zero, such that
\begin{equation}\label{make it short}
\norm{
a_{n,m}(t)\pa{\begin{array}{ll}
I_n& A\\
0&I_m
\end{array}}\pa{\begin{array}{ll}
\bp\\
\bq
\end{array}}
}_2<\del.
\end{equation}
Note that it follows that $\bq$ must be nonzero. The vector composed of the first $n$ coordinates of the vector appearing in ~\eqref{make it short} equals $e^{mt}(A\bq+\bp)$, while the one composed
of the last $m$ coordinate equals $e^{-nt}\norm{\bq}_2$. Both these vectors have length $<\del$, hence we conclude that $\norm{A\bq+\bp}_2=\idist{A\bq}$, and that
\begin{align}
\label{first}\idist{A\bq}&\le\del e^{-mt},\\
\label{second}0<\norm{\bq}_2&\le \del e^{nt}.
\end{align}
Let $N>e^{nT_0}$ be given and choose $t>T_0$ so that $N=e^{nt}$. We conclude from the above that there exist a nonzero $\bq\in\bZ^m$ satisfying $\norm{\bq}_2<\del e^{nt}<N$, 
$\idist{A\bq}\le\frac{\del}{e^{mt}}=\frac{\del}{N^\frac{m}{n}}$, and so $A$ is singular as desired. 

We leave the  other implication as an exercise to the interested reader, as we shall only use in this paper the implication proved above, i.e.\ that if $A$ is a nonsingular form, then $H_{n,m}x_A$ is nondivergent.  
\end{proof} 
Before stating Theorem \ref{app 4}, we introduce some more notation. Let $p:\bR^d\to\bR^n$ be the projection $p(\bx,\by)=\bx$. Abusing the notation introduced in \S\ref{H-S-T}, we denote for a subspace $U<\bR^d$ and a vector $w_0\in\bR^d$, by $\lam_{w_0+U}$, the natural $U$-invariant measure
supported on the affine subspace $w_0+U$.
\begin{theorem}\label{app 4}
For any nonsingular form $A\in\on{Mat}_{n\times m}(\bR)$, and for any measure of the form $\mu= p_*(\lam_{w_0+U})$, where $U$ is a nontrivial subspace of $\bR^d$, rational
with respect to the lattice $x_A$ , and $w_0$ is arbitrary, for $\mu$-almost any $\bx\in\bR^n$,
\begin{equation}\label{dense set 2}
\textrm{the set }\;\set{\norm{\bq}_2^m\idist{A\bq+\bx}^n:\bq\in\bZ^m}\textrm{ is dense in }[0,\infty).
\end{equation}
\end{theorem} 
\begin{proof}
Theorem
~\ref{rank 1} together with Lemma ~\ref{n-d} imply that $x_A$ is $\lam_{w_0+U}$-a.s.\ grid-DV. This means that for $\lam_{w_0+U}$-almost any $w=(\bx,\by)^t\in \bR^d$, the grid $x_A+w$
is DV. Calculating the value set we get 
\begin{equation}\label{V_F}
V_{P_{n,m}}(x_A+w)=\set{\norm{\bq+\by}_2^m\norm{A\bq+\bx+\bp}_2^n: \bp\in\bZ^n,\bq\in\bZ^m}.
\end{equation}
Given $w=(\bx,\by)^t$ such that $x_A+w$ is DV, for any $\ga\in[0,\infty)$, there are appropriate sequences $\bp_i,\bq_i$ satisfying 
$\lim_i\norm{\bq_i+\by}^m\norm{A\bq_i+\bx+\bp_i}^n=\ga$ (and the convergence is not trivial). Then $\norm{\bq_i}_2$ must go to infinity, which in turn implies that $\norm{A\bq_i+\bx+\bp_i}_2\to 0$ and in particular, $\norm{A\bq_i+\bx+\bp_i}_2=\idist{A\bq_i+\bx}$. It follows that $\lim_i\norm{\bq_i}^m\idist{A\bq_i+\bx}^n=\ga$ as well. We conclude that 
$$p^{-1}\pa{\set{\bx\in\bR^n:\textrm{ ~\eqref{dense set 2} dose not hold}}}\subset\set{w=(\bx,\by)^t\in\bR^d:\textrm{ the grid }x_A+w\textrm{ is not DV}}.$$
As the right hand side is $\lam_{w_0+U}$-null, the left hand side is too and by definition, for $p_*(\lam_{w_0+U})$-almost any $\bx$, ~\eqref{dense set 2} holds.
\end{proof}
The above theorem, when applied with $U=\bR^d$, gives Theorem ~\ref{simple version}, as mentioned earlier.
Let us demonstrate the strength of Theorem ~\ref{app 4} with the following
\begin{corollary}\label{cor 4}
Take $m=1$ and $n$ arbitrary. For any nonsingular vector $v\in\bR^n$ and for Lebesgue almost any $t\in \bR$ the following set 
$$\set{\av{q}\idist{(q+t)v}^n:q\in\bZ},$$
is dense in $[0,\infty)$.
\end{corollary}
\begin{proof}
Apply Theorem ~\ref{app 4} with $w_0=0$ and $U$ being the one dimensional subspace $$U=\set{t\pa{\begin{array}{ll}v\\1\end{array}}\in\bR^{n+1}: t\in\bR}.$$
\end{proof}
\section{The Mixing and the Coset lemmas}\label{tools}
In this section we study the following question. Given a sequence of endomorphisms of the $d$-torus, $\ga_n:\bT^d\to\bT^d$, and a Haar measure, $\lam_U$, of a subtorus (see the notation of \S\ref{H-S-T}), what can one say about the closure $\overline{\set{\ga_nv:n\in\bN}}$, for $\lam_U$-almost any point $v\in\bT^d$. We shall
see in the Coset lemma below, that unless an obvious obstacle is present, this closure must contain a coset of a nontrivial subtorus. In the course
of proving
the Coset lemma, which is the goal of this section, we shall prove the Mixing lemma, stated below, which is of independent interest on its own.
Both the Coset and the Mixing lemmas are stated in somewhat greater generality than we actually need in practice. This is dictated by the argument we use to derive
the Coset lemma, which is inductive, hence we need to let the dimensions of the tori to be arbitrary.
 
Given a compact metric space $X$, the space of Borel probability measures on it, $\cP(X)$, is compact with respect to the weak$^*$ topology.
When considering measures, we shall only refer to the weak$^*$ topology. If $\ga:X\to Y$ is a measurable map between compact metric spaces, it defines
a map $\ga_*:\cP(X)\to\cP(Y)$, given by $\ga_*(\eta)(A)=\eta(\ga^{-1}(A))$, for any Borel measurable $A\subset Y$ and $\eta\in\cP(X)$. The following 
definition is new as far as we know.  
\begin{definition}\label{m.s}
Given two probability measures $\eta\in\cP(X),\nu\in \cP(Y)$, and a sequence of measurable maps $\ga_n:X\to Y$, we say that $\ga_n$ \textit{mixes} $\eta$ \textit{to} $\nu$, if for any absolutely continuous probability measure $\eta'\ll\eta$, the sequence $(\ga_n)_*\eta'\to
\nu$ weak$^*$ (we sometime say that $(\ga_n)_*\eta$ converges mixingly to $\nu$).
\end{definition}
Note that the above definition is equivalent to the requirement that for any $g\in L^1(\eta)$, and any continuous $f\in C(Y)$ one has
\begin{equation}\label{mix 1}
\int_Xf(\ga_n(x))g(x)d\eta(x)\to \int_Yf(y)d\nu(y)\int_Xg(x)d\eta(x).
\end{equation}
To explain the terminology, recall that given $\nu\in\cP(X)$ and a sequence of measurable maps $\ga_n:X\to X$, preserving $\nu$,
the sequence $\ga_n$ is said to be a mixing sequence with respect $\nu$,
if for any two measurable sets $A,B\subset X$, one has $\nu(\ga_n(A)\cap B)\to \nu(A)\nu(B)$. This is easily seen to be equivalent to the fact that 
$\ga_n$ mixes $\nu$ to itself using the terminology introduced above. 

We note here that the fact that $\ga_n$ mixes $\eta$ to $\nu$ is much stronger than the fact that $(\ga_n)_*\eta\to \nu$. In particular, one consequence that we will be interested in is the following simple lemma.
\begin{lemma}\label{mixing implies 1}
If $\ga_n:X\to Y$ mixes $\eta$ to $\nu$, then for $\eta$-almost any $x\in X$, the closure $\overline{\set{\ga_nx:n\in\bN}}$ contains $\on{supp}(\nu)$.
\end{lemma}  
\begin{proof}
Assume that the conclusion is false. Then, there must exist a bump function $f\in C(Y)$ (i.e.\ $0\le f\le 1$), with $\int_Yfd\nu>0$, and a measurable set $\Om\subset X$, with $\eta(\Om)>0$, such that for any $x\in \Om$, the sequence $\ga_nx$ never visits $U=\set{y\in Y:f(y)>0}$. Let $g\in L^1(\eta)$ be the characteristic function of $\Om$. We now see that the left hand side of ~\eqref{mix 1} is constantly zero, while the right hand side is positive. A contradiction. 
\end{proof}
\begin{lemma}[Mixing lemma]
Let $U<\bR^{d_1}$ be a subspace, rational with respect to $\bZ^{d_1}$. Let $\ga_n:\bT^{d_1}\to\bT^{d_2}$ be a sequence of homomorphisms induced by the matrices $\ga_n\in \on{Mat}_{d_1\times d_2}(\bZ)$. Assume the sequence $(\ga_n)_*\lam_U$ converges weak$^*$ to $\mu\in\cP(\bT^{d_2})$.
Then, there exists a rational  subspace $V<\bR^{d_2}$, such that $\mu=\lam_V$, and furthermore, the sequence $\ga_n$ mixes $\lam_U$ to $\lam_V$, if and only if, there are no integer vectors $\bm\in\bZ^{d_2}$ and $0\ne \bk\in U$, satisfying $\ga_n^t\bm+\bk\in U^\perp$, for infinitely many $n$'s.
\end{lemma}
Before turning to the proof we make some clarifying remarks and work out some examples. First, the fact that $\mu$ is a Haar measure is not the essence of the lemma. It is an
easy exercise to prove that any weak$^*$ limit of Haar measures is again a Haar measure in this context (this fact is implicitly used in the proof below).
The true content of the lemma is the fact that in the absence of the obvious obstacle to mixing -- the existence of $\bm,\bk\ne 0$ such that 
$\ga_n^t\bm+\bk\in U^\perp$ infinitely often -- $\ga_n$ actually mixes $\lam_U$ to $\lam_V$, which is significantly stronger than the convergence
of $(\ga_n)_*\lam_U$ to $\lam_V$. 
Second, note that given any sequence $\ga_n:\bT^{d_1}\to \bT^{d_2}$, of homomorphisms,
and a Haar measure $\lam_U$ of a subtorus $U+\bZ^{d_1}\subset \bT^{d_1}$, one can always assume, after passing to a subsequence, that $(\ga_n)_*\lam_U$
converges  by the compactness of $\cP(\bT^{d_2})$. Third, note that one could rephrase the condition ensuring mixing, as saying that for
any $\bm\in \bZ^{d_2}$, either $\ga_n^t\bm$ is eventually in $U^\perp$, or the distance from $\ga_n^t\bm$ to $U^\perp$ goes to $\infty$. Finally we note
that when $d_1=d_2=2$, $U=\bR^{d_1}$, and the $\ga_n$'s are automorphisms, the lemma is proved in ~\cite{Bergelson-Gorodnik} Lemma 2.2. 
\begin{example}\label{example 1}
Let $\ga\in\on{GL}_d(\bZ)$ be an automorphism of the $d$-torus which has an irreducible characteristic polynomial all of whose roots are real. Then for any subtorus $U+\bZ^{d}$, the sequence $\ga^n_*\lam_U$ converges mixingly to the Haar measure of the full $d$-torus. In particular, for $\lam_U$ almost any point $v\in\bT^d$, the orbit $\set{\ga^nv}$, is dense in $\bT^d$.

We split the argument into two parts. In the first we show that if a subsequence $\ga^{n_i}_*\lam_U$ converges to some $\lam_V$, then $V=\bR^d$, which proves that $\ga^n\lam_U$ converges to $\lam_{\bR^d}$. In the second we show that the convergence is mixing. We rely on the following two properties of $\ga$ (resp.\ $\ga^t$):
\begin{enumerate}
\item $\ga$ (resp.\ $\ga^t$) is diagonalizable over $\bR$ with all eigenvalues distinct and of absolute value $\ne 1$. Hence, for any $v\in\bR^d$, $\ga^nv$ converges projectively to a one dimensional eigenspace of $\ga$. 
\item Any $\ga$-invariant (resp.\ $\ga^t$-invariant) subspace (which is necessarily a direct sum of eigenspaces), is not a subspace of any proper rational subspace (this follows from the irreducibility of the characteristic polynomial).\\
\end{enumerate}
\textbf{First step}: If $n_i$ is such that $\ga^{n_i}_*\lam_U$ converges to some $\lam_V$, then in particular, for any $v\in U$, any projective limit of $\ga^{n_i}v$, is a line in $V$. But, from (1) it follows that such a line is an eigenspace of $\ga$ and property (2) implies that $V$, which is a rational subspace, must equal $\bR^d$.\\
\textbf{Second step}: We need to show that there could not exist
integer vectors $\bm\in \bZ^d,0\ne \bk\in U$, such that $(\ga^t)^n\bm+\bk\in U^\perp$ for infinitely many $n$'s.  From property (1) we deduce that $(\ga^t)^n\bm$ converges projectively
to a 1-dimensional eigenspace of $\ga^t$ which is contained in $U^\perp$ (as $\bk$ is fixed and $(\ga^t)^n\bm$ must diverge), which is a rational proper subspace (as $U\ne\set{0}$). This contradicts property (2).
\end{example}
\begin{example}\label{example 2}
Let $\ga_n=\pa{\begin{array}{ll}
1&0\\
n&1
\end{array}}$. If $U=\set{(x,x)\in\bR^2: x\in\bR}$, then although $(\ga_n)_*\lam_U$ converges to the Haar measure of the full 2-torus, $\ga_n$ does
not mix $\lam_U$ to $\lam_{\bR^2}$. This is because, if we choose $\bm=(1,0)^t$, then $\bm$ is a fixed point of $\ga_n^t$, and is not in $U^\perp$, so we 
see that $\ga_n^t\bm$ is not eventually in $U^\perp$ and its distance from $U^\perp$ does not diverge. Indeed, Lemma ~\ref{mixing implies 1} does not apply and for $\lam_U$-almost any $v\in\bT^2$, the closure $\overline{\set{\ga_nv}}$ equals a coset of a lower dimensional subtorus. 
\end{example}
\begin{proof}[Proof of the Mixing lemma] 
For any $d$ and $\bm\in \bZ^d$, let $e_\bm(v)=e^{2\pi i<\bm,v>}:\bT^d\to\bC$. Note that in our notation there is no reference to the dimension and the reader should understand from the context what is the domain of the character $e_\bm$. In particular, there will be times in which
in one equation, two dimensions will be mixed. We need to find a rational subspace $V<\bR^{d_2}$, such that ~\eqref{mix 1} holds for any $f\in C(\bT^{d_2}), g\in L^1(\lam_U)$. A standard argument shows 
that it is enough to verify the validity of ~\eqref{mix 1} when $f$ and $g$ are chosen from sets spanning dense subspaces of 
$C(\bT^{d_2})$ and $L^1(\lam_U)$ respectively. From the Stone-Weierstrass theorem it follows that $\set{e_\bm:\bm\in\bZ^{d_2}}$ spans a dense subspace
of $C(\bT^{d_2})$. Also, the set $\set{e_\bk:\bk\in U}$ spans a dense subspace in $L^1(\lam_U)$. This is because it forms an orthonormal basis to $L^2(\lam_U)$, which is dense in $L^1(\lam_U)$ by the Cauchy Schwarz inequality.

We conclude that we need to find a rational subspace $V<\bR^{d_2}$, such that the following convergence holds
\begin{align}
\label{mix 2} \textrm{For any }  \bm\in\bZ^{d_2},\bk&\in U,\;\;  \int_{\bT^{d_1}}e_{\ga_n^t\bm+\bk}(v)d\lam_U(v)\\
\label{mix 3}  &=\int_{\bT^{d_1}}e^{2\pi i <\ga_n^t\bm+\bk,v>}d\lam_U(v)=\int_{\bT^{d_1}}e_\bm(\ga_nv)e_k(v)d\lam_U(v)\\
\label{mix 4} &\to\int_{\bT^{d_2}}e_\bm(v)d\lam_V(v)\int_{\bT^{d_1}}e_\bk(v)d\lam_U(v).
\end{align}
Given a rational subspace $V<\bR^d$ and a corresponding Haar measure $\lam_V$, a short calculation shows that for any $\bm\in\bZ^d$ we have
\begin{equation}\label{haar measures}
\int e_\bm(v)d\lam_V(v)=\bigg\{
\begin{array}{ll}
1& \textrm{ if }\bm\in V^\perp,\\
0& \textrm{ otherwise}.
\end{array}
\end{equation}
Working with equation ~\eqref{haar measures}, we conclude that for any choice of $V$, the values in ~\eqref{mix 2},~\eqref{mix 4} satisfy
\begin{equation*}
~\eqref{mix 2}=\bigg\{
\begin{array}{ll}
1& \textrm{ if }\ga_n^t\bm+\bk\in U^\perp,\\
0& \textrm{ otherwise}
\end{array}\;;\;
~\eqref{mix 4}=\bigg\{
\begin{array}{ll}
1& \textrm{ if }\bm\in V^\perp\textrm{ and }\bk=0,\\
0& \textrm{ otherwise}
\end{array}
\end{equation*}  
We conclude that given $\lam_V$, the sequence $\ga_n$ mixes $\lam_U$ to $\lam_V$ if and only if the values of ~\eqref{mix 2} and ~\eqref{mix 4}
agree for all large enough $n$'s. This implies immediately the only if part of the lemma; if there exist $\bm\in\bZ^{d_2},0\ne \bk\in U$, such that $\ga_n^t\bm+\bk\in U^\perp$ for infinitely many $n$'s, then for these values of $n$, the value of ~\eqref{mix 2} is 1, while the value
of ~\eqref{mix 4} is 0. 

We are left to prove the if part. Assume then that for any $\bm\in \bZ^{d_2}$ and $0\ne \bk\in U$, $\ga_n^t\bm+\bk$ is eventually outside $U^\perp$.
We shall conclude the proof by showing that for any subsequence of $\ga_n$, there exists yet another subsequence $\tilde{\ga}_n$, and a subspace $V$, 
such that $\tilde{\ga}_n$ mixes $\lam_U$ to $\lam_V$. To see why this concludes the proof, note first that this implies that $\mu=\lam_V$, as we assume that $(\ga_n)_*\lam_U$ converges to $\mu$, hence any subsequence of it converges to $\mu$ too. In particular, this shows that $V$ does not depend on the 
initial subsequence. Second, if $\ga_n$ does not mix $\lam_U$ to $\lam_V$, then the above discussion shows that there must exist integer vectors
$\bm\in\bZ^d$ and $\bk\in U$, such that for infinitely many $n$'s, the value in ~\eqref{mix 2} is different from the value in ~\eqref{mix 4}; we obtain a 
contradiction once we take as the initial subsequence, those $\ga_n$'s for which the values in ~\eqref{mix 2} and ~\eqref{mix 4} do not agree, seeing that there could not be a subsequence $\tilde{\ga}_n$ of that subsequence, which mixes $\lam_U$ to $\lam_V$.

To this end, let a subsequence of $\ga_n$ be given. By a standard diagonal argument, it has a subsequence $\tilde{\ga}_n$ such that for any $\bm\in\bZ^{d_2}$, either $\tilde{\ga}_n^t\bm\in U^\perp$ for all but finitely many $n$'s, or $\tilde{\ga}_n^t\bm\notin U^\perp$ for all but finitely many $n$'s.
Define 
\begin{align}\label{V 1}
V^\perp&=\on{span}\set{\bm\in\bZ^{d_2}: \tilde{\ga}_n^t\bm\in U^\perp \textrm{ for all large }n's}\\
 \nonumber &=\on{span}\set{\bm\in\bZ^{d_2}: \tilde{\ga}_n^t\bm\in U^\perp \textrm{ for infinitely many }n's}.
\end{align}
We remark	 that the above set of integer vectors spanning $V^\perp$ is clearly a group and in fact it is the intersection of $V^\perp$ with $\bZ^{d_2}$; that is,
if $\bm\in V^\perp$ is an integer vector, then $\tilde{\ga}_n^t\bm\in U^\perp$ for all large enough $n$'s.

We now check that along the subsequence $\tilde{\ga}_n$, the values of ~\eqref{mix 2},~\eqref{mix 4} agree for all large enough $n$'s, hence concluding that $\tilde{\ga}_n$ mixes $\lam_U$ to $\lam_V$ as desired. There are two cases in which the values in ~\eqref{mix 2},~\eqref{mix 4} do not 
agree (with $\ga_n$ replaced by $\tilde{\ga}_n$). 
\begin{enumerate}
	\item\label{mix 5} Either $\tilde{\ga}_n^t\bm+\bk\notin U^\perp$ and $\bk=0$ and $\bm\in V^\perp$,
	\item\label{mix 6} or $\tilde{\ga}_n^t\bm+\bk\in U^\perp$ but either $\bk\ne 0$ or $\bm\notin V^\perp$.
\end{enumerate}
It follows from the remark made after the definition of $V^\perp$, that ~\eqref{mix 5} cannot happen infinitely many times. Assume ~\eqref{mix 6}
holds for infinitely many $n$'s. It cannot be that $\bk\ne 0$ because of our assumption that $\ga_n^t\bm+\bk$ is eventually outside $U^\perp$, for any $\bm$ and nonzero $\bk$. On the other hand, it cannot be that $\bk=0$, by our construction of $\tilde{\ga}_n$ and the definition of $V^\perp$, as it would mean that eventually $\tilde{\ga}_n^t\bm\in U^\perp$, hence $\bm\in V^\perp$. In any case we arrive to a contradiction as desired.
\end{proof}
We now use the Mixing lemma to prove another lemma which serves as the main tool used to prove Theorem ~\ref{main theorem}. As in the case of the mixing
lemma, this lemma is of independent interest on its own.
\begin{lemma}[Coset lemma]\label{c.l}
Let $U<\bR^{d_1}$ be a subspace, rational with respect to $\bZ^{d_1}$. Let $\ga_n:\bT^{d_1}\to\bT^{d_2}$ be a sequence of homomorphisms induced by the matrices $\ga_n\in \on{Mat}_{d_1\times d_2}(\bZ)$, such that the restrictions $\ga_n|_U$ form an infinite set.
Then there exists a rational subspace $\set{0}\ne V<\bR^{d_2}$, such that for $\lam_U$-almost any $v\in\bT^{d_1}$, the closure $\overline{\set{\ga_nv:i\in \bN}}$ contains a coset of the subtorus $V+\bZ^{d_2}$
\end{lemma} 
Note that in the above lemma, the subtorus $V+\bZ^{d_2}$ is fixed, while its coset depend on the initial point.
\begin{proof}
Along the argument we will take subsequences a large number of times. 
We will abuse notation and continue to denote the subsequences by the same symbols.
The proof goes by induction on $d_2$ but we first make some general observations.
First, by taking a subsequence, we may assume that the restrictions $\ga_n|_U$ are distinct and that $(\ga_n)_*\lam_U$ converges. 
If there are no integer vectors $\bm\in \bZ^{d_2},0\ne \bk\in U$, such that
$\ga_n^t\bm+\bk \in U^\perp$ for infinitely many $n$'s, then the Mixing lemma implies that $\ga_n$ mixes $\lam_U$ to some $\lam_V$. Of course
$V\ne\set{0}$, as equality would imply that $U=\set{0}$ which contradicts our assumption that the restrictions $\ga_n|_U$ are distinct. Lemma 
~\ref{mixing implies 1} now tells us that for $\lam_U$-almost any $v\in\bT^{d_1}$, the closure $\overline{\set{\ga_n v}}$ actually contains the 
subtorus $V+\bZ^{d_2}$ and the lemma follows with the additional information that the coset we gain in the closure does not depend on the initial point. Assume now that there exits integer vectors $\bm\in\bZ^{d_2}, 0\ne \bk\in U$, such that $\ga_n^t\bm+\bk\in U^\perp$ for infinitely many $n$'s. By passing to a subsequence we may assume this happens for all $n$. Note that the validity of the lemma is stable under a change of variable, that is if $A\in\SL_{d_1}(\bZ),B\in\SL_{d_2}(\bZ)$ are two 
automorphisms of $\bT^{d_i}$, then the lemma is true for $U,\ga_n$, if and only if it is true for $A^{-1}U,B\ga_nA$. It follows that we might assume that $U$ is the standard embedding of $\bR^r$ in $\bR^{d_1}$ for some $1\le r\le d_1$ (i.e.\ $U$ is the subspace corresponding to the first $r$ coordinates), and that $\bk\in U$ is colinear to $e_1$, i.e.\ there exists an integer $\ell$ such that 
$\bk=-\ell e_1=(-\ell,0,\dots,0)^t\in\bZ^{d_1}$. This is done by choosing $A$ properly.
By choosing $B$ properly, we see that we may assume that $\bm=e_1$. Indeed, $\bm$ was chosen to be an integer vector satisfying $\ga_n^t\bm+\bk\in U^\perp.$ Any choice of $B$ results in replacing the original $\bm$ by $(B^t)^{-1}\bm$. Choosing $B$ properly, we can guarantee that $\bm$ is an integer multiple of 
$e_1$, but as it can be chosen to be primitive, we may assume it actually equals $e_1$. To summarize, if we assume that $\ga_n$ does not mix $\lam_U$
to some $\lam_V$, we may assume we are in the following position: 
\begin{enumerate}
	\item The space $U$ equals the standard copy of $\bR^r$ in $\bR^{d_1}$.
	\item The matrices $\ga_n$ satisfy $\ga_n^t e_1-\ell e_1 \in U^\perp$. This means that  the first row of the matrix
	$\ga_n$ is of the form $(\underbrace{\ell,0,\dots,0}_{r},*\dots *)$.
\end{enumerate} 
The particular shape of the first row of the $\ga_n$'s is what enables us to have a reduction to a lower dimension and invoke the inductive hypothesis.
We now describe the proof by induction.\\
\textbf{The case} $d_2=1$: We prove this case by induction on $d_1$. Note that in this case there is just one possibility for $V$, and only one coset;
 this means that we need to prove that for $\lam_U$-almost any $v\in\bT^{d_1}$, the sequence $\ga_nv$, is dense in $\bT^1$. For $d_1=1$, $\ga_n$ are simply integer numbers and $U=\bR$. The fact that we assume that $\ga_n$ are all distinct imply that $\ga_n$ always mixes $\lam_U$ to itself and we are done as explained above. Assume we know the lemma
for $d<d_1$ and $d_2=1$. If $\ga_n$ mixes $\lam_U$ then we are done as explained above. If not, then we may assume after a change of variable, 
as explained above, that 
the first coordinate of the matrix $\ga_n$ (which is a row vector), is fixed and equals to some integer $\ell$. In particular, this forces the dimension of $U$ to be $\ge 2$, as we assume that the restrictions $\ga_n|_U$ are all distinct. The lemma now clearly follows from the
inductive hypothesis, as given a vector $v=(v_1,\dots,v_r,0\dots 0)^t\in U$, we know that for almost any $v'=(v_2,\dots,v_r,0\dotso 0)^t$, the sequence $\ga'_nv'$ is dense in $\bT^1$ (here $\ga_n'$ are the $\ga_n$'s after we erase the first coordinate). But then we see
that for $\lam_U$-almost any $v\in U$, the sequence $\ga_nv=\ell v_1+\ga_n'v'$ is dense in $\bT^1$. This finishes the case $d_2=1$ which is the base of
our induction.\\
\textbf{The inductive step}: Assume we know the lemma when the dimension of the image torus is less than $d_2$. If $\ga_n$ mixes $\lam_U$ then we are done as explained above. If not, then after a change of variable we may assume that the $\ga_n$'s are of the form
\begin{equation}\label{matrix dec}
\ga_n=\pa{\begin{array}{ll}
\ell& s_n\\
p_n&q_n
\end{array}},
\end{equation}
where $p_n$ is a column vector, $s_n=(\underbrace{0,\dots,0}_{r-1},*\dots *)$, and $q_n$ are matrices of the appropriate dimension. Let us denote 
by $\ga_n'$ the matrices obtained from $\ga_n$ by erasing the first row. We view them as homomorphisms from $\bT^{d_1}$ to $\bT^{d_2-1}$. The restrictions of $\ga_n'$ to $U$ (which is the copy of $\bR^r$ in $\bR^{d_1}$) must all be distinct, as we know this for the $\ga_n$'s, and their 
first row is identical, as far as $U$ is concerned. From the inductive hypothesis we conclude that there exists a nontrivial subtorus $V+\bZ^{d_2-1}$ in $\bT^{d_2-1}$, such that for $\lam_U$-almost any $v\in \bT^{d_1}$, the closure $\overline{\set{\ga_n'v}}$ contains a coset, $w'_v+(V+\bZ^{d_2-1})$, where $w'_v\in\bR^{d_2-1}$.
We conclude that for such $v$'s, the closure $\overline{\set{\ga_nv}}$ contains the coset $w_v+(V+\bZ^{d_2})$, where $w_v$ is obtained from $w_v'$
by adding to it as a first coordinate $\ell v_1$. This concludes the proof.
\end{proof}
We work out a few examples to develop a feeling about the null set of points which corresponds to closures which do not contain a nontrivial coset.
A first example is given in Example ~\ref{example 2}, where the points $v=(x,x)^t\in U$ for which the closure of the sequence $\ga_nv=(x,(n+1)x)$,
in $\bT^2$, does not contain a coset of the subtorus corresponding to the second coordinate, are exactly those with $x\in\bQ$, which clearly form a
$\lam_U$-null set. As a slightly more complicated example we have
\begin{example}\label{example 3}
Let Let $\ga_n=\pa{\begin{array}{ll}
1&0\\
n&n
\end{array}}$, and $U=\bR^2$. Then for $v=(x,y)^t$, the closure of the sequence $\ga_nv=(x,n(x+y))$ equals a coset of the subtorus corresponding to 
the second coordinate, if and only if $v$ does not lie on a line of the form $\set{v:x+y=q}$ for some $q\in\bQ$. This countable collection of lines
is clearly a $\lam_U$-null set but still a dense set of dimension 1.
\end{example}
\begin{example}
Let $M$ be a diagonalizable epimorphism of the $d$-torus. When we apply the Coset lemma to the case where $\ga_n=M^n$, it follows that for Lebesgue almost any $v$, the sequence $M^nv$ contains in its closure a full coset of some nontrivial subtorus. Recently it was proved in ~\cite{BFK} (see also ~\cite{BFKW}), that for any 
given $w$, the set $\set{v\in\bT^d: \inf_n\on{d}(M^nv,w)>0}$ is dense and of full dimension. In fact, they show it is a winning set for Schmidt's game  (this generalizes Dani's work ~\cite{Dani-endo-of-tori}). In particular,
the set of points $v\in\bT^d$ which violate the conclusion of the Coset lemma, although is null, is still large from other perspectives. 
\end{example}
\section{Proofs of Theorem \ref{main theorem} and Corollary ~\ref{product 1}}\label{pfs}
We begin by rephrasing the Coset lemma, when we let the lattice vary. We shall use the following notation: For a grid $y=x+w\in Y_d$ and $v\in\bR^d$,
we denote by $y+v$, the grid $x+(w+v)$. Given a subspace $V<\bR^d$, we let $y+V=\set{y+v:v\in V}\subset \pi^{-1}(x)$. Thus $y+V$ is simply the coset of 
$V$ passing through the grid $y$, in the torus $\pi^{-1}(x)$.
\begin{lemma}\label{coset 2}
Let $x_1,x_2\in X_d$ be two lattices and $h_n\in G$ be such that $h_nx_1\to x_2$. Let $U<\bR^d$ be a rational subspace with respect to $x_1$ and $w\in\bR^d$ be given. Assume
that $h_n$ is not almost finite with respect to $U$ (in the sense of Definition ~\ref{a.f}). Then, there exists a subspace $\set{0}\ne V<\bR^d$, rational with respect to $x_2$, such that for $\lam_{w+U}$-almost any grid $y\in\pi^{-1}(x_1)$, the closure $\overline{\set{h_ny:n\in\bN}}$,
contains a coset $z_y+V$, for some $z_y\in\pi^{-1}(x_2)$.
\end{lemma}
\begin{proof}
We first note that after passing to a subsequence, we may assume that $h_n(w+x_1)$ converges to some grid in $\pi^{-1}(x_2)$. This clearly reduces
the lemma to the case where we are dealing with a Haar measure of a subtorus $\lam_U$, rather than a translate of it $\lam_{w+U}$.

Let $g_i\in G$ be such that $x_i=g_i\Ga$ for $i=1,2$. There exists a sequence $\eps_n\in G$, with $\eps_n\to e$ such that $\eps_nh_nx_1=x_2$, or 
in other words, $\ga_n=g_2^{-1}\eps_nh_ng_1\in \Ga$. Let $U'=g_1^{-1}U$. Then $U'$ is rational with respect to $\bZ^d$. The fact that $h_n$ is not 
almost finite with respect to $U$, translates to the fact that the restrictions $\ga_n|_{U'}$ form an infinite set. Hence, the Coset lemma applies and gives us the existence of a subspace $\set{0}\ne V'<\bR^d$, rational with respect to $\bZ^d$, such that for $\lam_{U'}$-almost any grid $y'$ of $\bZ^d$, the closure of the sequence $\ga_ny'$ contains a coset of the subtorus $V'+\bZ^d$. This translates to the fact that for $V=g_2V'$, and $\lam_U$-almost any
grid $y\in\pi^{-1}(x_1)$, the closure of the sequence $\eps_nh_ny$, contains a coset of the subtorus $V+x_2$. This implies the same statement for the sequence $h_ny$, as $\eps_n\to e$, and the lemma follows.   
\end{proof}
Before turning to the proof of Theorem ~\ref{main theorem}, we note that in the terminology introduced at the beginning of this section, we have that a continuous function $F:\bR^d\to\bR$ is nondegenerate in the sense of Definition ~\ref{nondeg}, if and only if, for any grid $y\in Y_d$ and any subspace $\set{0}\ne V<\bR^d$ one has $\cup_{y'\in\set{ y+V}}V_F(y')=F(\bR^d)$.
\begin{proof}[Proof of Theorem ~\ref{main theorem}]
Let $F,\; x,\; U,\;w$, and $h_n$, be as in the statement. Denote $x=x_1$ and $\lim h_n x_1=x_2$. Lemma ~\ref{coset 2} implies that there exists a nontrivial subspace $V<\bR^d$, such that for $\lam_{w+U}$-almost any $y\in\pi^{-1}(x_1)$, the closure  $\overline{H_Fy}$, contains a full coset 
$z_y+V$, for $z_y\in\pi^{-1}(x_2)$. The Inheritance lemma (Lemma ~\ref{inheritance}) implies now that such a grid satisfies
$$V_F(y)\supset\cup_{y'\in \set{z_y+V}} V_F(y')=F(\bR^d),$$
which concludes the proof.
\end{proof}
\begin{proof}[Proof of Corollary ~\ref{product 1}]
Let $x$, $U$, and $w$, be as in the statement. From the classification of divergent $H_F$-orbits (see ~\cite{TW} and \S\ref{main results}), we see that as $x$ does not contain any vectors on the axes, it has a nondivergent $H_F$-orbit. Let $h_n\in H_F$ be a diverging sequence such that $h_nx$
converges. We only need to argue why the sequence $h_n$ cannot be almost finite with respect to $U$. Indeed, assume that $h_n$ is almost finite with respect to $U$. After passing to a subsequence and possibly permuting the coordinates (which we ignore in order to ease our notation), we can assume
that $h_n=\diag{e^{t_1^{(n)}},\dots,e^{t_d^{(n)}}}$, where for some $1\le r\le d$ we have that $t_i^{(n)}$ diverge for $i\le r$ and converge for $i>r$.
It follows that as $h_n$ is almost finite with respect to $U$, we must have $U<\set{(\underbrace{0,\dots,0}_r,*\dots *)^t\in\bR^d}$. This leads to a 
contradiction, as $U$ is rational with respect to $x$, and we assumed that $x$ does not contain any point on the hyperplanes of the axes. 
\end{proof}
\begin{acknowledgments}
I would like to express my gratitude to many people who contributed to this paper in one way or another. Some by stimulating conversations, others by
sharing references. Thanks are due to Hillel Furstenberg, Barak Weiss, Elon Lindenstrauss, Manfred Einsiedler, Alexander Gorodnik, Dmitry Kleinbock,
Yann Bugeaud, and Fabrizio Polo.
\end{acknowledgments}
\def\cprime{$'$}

\end{document}